\documentclass[12pt,oneside]{article}
\usepackage{amsmath,amssymb,amsfonts,amsthm}
\usepackage{color}
\usepackage[all]{xy}
\usepackage{graphicx}
\usepackage{color}
\usepackage{makeidx}
\usepackage{multirow}
\usepackage{rotating}
\usepackage{pdfpages}
\usepackage{pdflscape}
\usepackage{tabularray}
\usepackage{multicol}
\usepackage{enumitem}
 \usepackage{hyperref}

\allowdisplaybreaks

\numberwithin{equation}{section} 
\numberwithin{figure}{section} 

\theoremstyle{plain}
\newtheorem*{theorem*}{Theorem}
\newtheorem{theorem}{Theorem}[section]
\newtheorem{lemma}[theorem]{Lemma}
\newtheorem{proposition}[theorem]{Proposition}

\newtheorem{remark}[theorem]{Remark}

\theoremstyle{definition}
\newtheorem{definition}[theorem]{Definition}

\theoremstyle{remark}
\newtheorem{example}{Example}
\newtheorem*{acknowledgement*}{Acknowledgement}

\numberwithin{equation}{section}

\newcommand\overcirc[1]{\raisebox{10pt}{\tiny{$\circ$}}{\kern-7.5pt}\mbox{$#1$}}
\newcommand\undersym[2]{\raisebox{-6pt}{$#2$}{\kern-5pt}\mbox{$#1$}}
\newcommand\overdiamond[1]{\raisebox{10pt}{\small$\star$}{\kern-7.5pt}\mbox{$#1$}}
\newcommand\overast[1]{\raisebox{10pt}{\small$\ast$}{\kern-7.5pt}\mbox{$#1$}}
\newcommand\overlind[1]{\raisebox{10pt}{\small$\overline{{\hspace{2pt}}\star}$}{\kern-7.5pt}\mbox{$#1$}}
\newcommand\overlinc[1]{\raisebox{10pt}{\small$\overline{{\hspace{2pt}}\circ}$}{\kern-7.5pt}\mbox{$#1$}}
\newcommand\overlina[1]{\raisebox{10pt}{\small$\overline{{\hspace{1pt}}\ast}$}{\kern-7.5pt}\mbox{$#1$}}

\begin{document}

\title{\bf The T-tensor of spherically symmetric Finsler metrics
}

 \author{ Salah G. Elgendi }
\date{}

\maketitle

\begin{center}
  Department of Mathematics, Faculty of Science,
  \\
  Islamic University of Madinah, Madinah, Saudi Arabia
\end{center}

\begin{center}
     salah.ali@fsc.bu.edu.eg , \ salahelgendi@yahoo.com
\end{center}

\vspace{0.3cm}

\begin{abstract}
This paper is devoted to the study of the T-tensor associated with  a spherically symmetric Finsler metric $F=u\phi(r,s)$ on \(\mathbb{R}^n\). We derive a general expression for the T-tensor in terms of the scalar function \(\phi(r, s)\) and its partial derivatives. Furthermore, we characterize all spherically symmetric Finsler metrics satisfying the so-called T-condition, that is, those for which the T-tensor vanishes. In addition, we obtain the formula for the mean Cartan tensor and demonstrate that all spherically symmetric Finsler metrics  of dimension $n \geq 3$, with a non-zero mean Cartan tensor are quasi-C-reducible.
\end{abstract}




\section{Introduction}
~\par 
Spherically symmetric Finsler metrics have attracted significant attention due to their geometric simplicity and wide applicability in both pure and applied contexts. A Finsler metric \(F\) on an open subset of \(\mathbb{R}^n\) is called spherically symmetric if it can be expressed in the form
\[
F(x, y) = u\,\phi(r, s), \quad r = |x|,\quad s = \frac{\langle x, y \rangle}{u}, \quad u=|y|
\]
where \(\phi\) is a smooth, positive function. These metrics naturally generalize the classical Riemannian notion of rotational symmetry and serve as key models in various physical theories, including generalizations of gravitational and cosmological frameworks \cite{Rutz1993, Pfeifer2012}. Their structure allows for the explicit computation of geometric invariants, making them ideal for analyzing curvature properties and variational behavior in Finsler geometry \cite{Zhou_Mo}.

\medskip

The T-tensor is one of the fundamental non-Riemannian curvature tensors in Finsler geometry. Introduced as a component of the curvature decomposition, it plays a crucial role in characterizing Landsberg and Berwald spaces \cite{Elgendi-LBp,Elgendi2021,Elgendi-ST_condition}. From a physical perspective, the T-tensor also emerges in the context of Finslerian extensions of general relativity, where it contributes to the anisotropic corrections of geodesic deviation and field equations \cite{Asanov1985, Vacaru2012}. Despite its significance, explicit computations of the T-tensor are rare and are usually confined to highly symmetric or special cases.

\bigskip

The structure of the paper is as follows. In Section 3, we investigate key geometric quantities associated with spherically symmetric Finsler metrics, with particular focus on the Cartan tensor and the mean Cartan tensor. These tensors play a central role in characterizing the non-Riemannian nature of Finsler spaces. We begin by deriving explicit expressions for these tensors in terms of the underlying metric function. Building on these computations, we establish a significant structural property: namely, that any spherically symmetric Finsler metric with a non-zero mean Cartan tensor and  of dimension $n \geq 3$  must be quasi-C-reducible.  In Section 4, we derive an explicit formula for the \(T\)-tensor associated with a spherically symmetric Finsler metric, and discuss several notable special cases. Section 5 is devoted to the study of the \(T\)-condition. In particular, we characterize all spherically symmetric Finsler metrics, of dimension $n \geq 3$, satisfying the \(T\)-condition, showing that they must take the form
\[
\phi(r,s) =
a(r)\, s^{\frac{c(r)\, r^2 - 1}{c(r)\, r^2}} ( r^2 -  s^2)^{\frac{1}{2 c(r) r^2}},
\]
where \(a(r)\) and \(c(r)\) are smooth functions of the radial variable \(r\). Two examples of spherically symmetric Finsler metric of Kropina-type and Randers-type are considered.

\section{Preliminaries} 
~\par 
   Spherically symmetric Finsler metrics constitute a notable class of Finsler structures characterized by rotational invariance around the origin. Defined on an open ball $\mathbb{B}^n(r_0) \subset \mathbb{R}^n$, such a metric can be written in the form
\[
F(x, y) = u\, \phi(r, s), \quad \text{where} \quad r = |x|,\; u = |y|,\; s = \frac{\langle x, y \rangle}{|y|},
\]
with $\phi$ a smooth positive function defined on $[0, r_0) \times \mathbb{R}$, and where $\langle \cdot, \cdot \rangle$ and $|\cdot|$ denote the standard Euclidean inner product and norm, respectively. 
 A spherically symmetric metric of the form \(F = u\,\phi(r,s)\), defined on the ball \(\mathbb{B}^n(r_0) \subset \mathbb{R}^n\), is regular if and only if the function \(\phi(r,s)\) is a positive, smooth function satisfying the following conditions:
\begin{equation*}
\label{Regular_condition}
\phi - s \phi_s > 0, \quad \phi - s \phi_s + (r^2 - s^2) \phi_{ss} > 0.
\end{equation*}

These metrics were studied extensively in \cite{MoZhou2009, MoZhouZhu2010}, where their geometric and curvature properties were analyzed in depth.

The spherical symmetry guarantees invariance under the orthogonal group $O(n)$, allowing for simplifications in the study of curvature conditions under rotational symmetry. Such metrics generalize radial Riemannian structures and are particularly useful in constructing explicit examples of Finsler spaces satisfying special geometric conditions, such as Landsberg or Berwald metrics \cite{Elgendi2021-SSM,Elgendi2023-SSM,Zhou_Mo,Tayebi-et.}.

Because of their symmetry, the associated geometric quantities often depend only on $r$ and $s$, this makes spherically symmetric Finsler metrics a valuable framework for the analytic investigation of Finsler geometry.

\bigskip

     Since the components of the metric tensor corresponding to the Euclidean norm are given by the Kronecker delta, i.e., $\delta_{ij}$, we can lower the indices of   $y^i$ and $x^i$ using:

    $$y_i := \delta_{ih} y^h, \quad     x_i := \delta_{ih} x^h.$$

Note that, due to the simplicity of the Euclidean metric, this lowering of indices does not alter the components; hence, $y_i = y^i$ and $x_i = x^i$. Therefore, in this paper,  we emphasize that $y_i \neq F\frac{\partial F}{\partial y^i}$; rather,  $y_i = u\frac{\partial u}{\partial y^i}$.
  
  \bigskip

The components $g_{ij}$ of the metric tensor of the spherically symmetric metric $F=u \phi(r,s)$ are given by  
\begin{equation}
\label{Eq:g_ij}
g_{ij}=\sigma_0\ \delta_{ij}+\sigma_1\  x_ix_j+\frac{\sigma_2}{u} (x_iy_j+x_jy_i)+\frac{\sigma_3}{u^2}y_iy_j,
\end{equation}
where  $$\sigma_0=\phi(\phi-s\phi_s),\quad \sigma_1= \phi_s^2+\phi\phi_{ss},\quad \sigma_2= (\phi-s\phi_s)\phi_s-s\phi\phi_{ss},  \quad \sigma_3= s^2\phi\phi_{ss}-s(\phi-s\phi_s)\phi_s.$$

 Throughout this paper, subscript  \(s\)   denotes partial derivative  with respect to the variable  \(s\).

 \medskip

The following geometric objects can be found, for example,  in \cite{Zhou_Mo}. 
The components $g^{ij}$ of the inverse metric tensor are given  as follows
\begin{align}
\label{Eq:g^ij}
g^{ij}=&\rho_0\delta^{ij}+ \frac{\rho_1}{u^2}y^iy^j+ \frac{\rho_2}{u} (x^iy^j+x^jy^i)+\rho_3x^ix^j,
\end{align}
where
$$
 \rho_0=\frac{1}{\phi(\phi-s\phi_s)}, \quad \rho_1=\frac{(s\phi+(r^2-s^2)\phi_s)(\phi\phi_s-s\phi_s^2-s\phi\phi_{ss})}{\phi^3(\phi-s\phi_s)(\phi-s\phi_s+(r^2-s^2)\phi_{ss})},$$
 $$\rho_2=-\frac{\phi\phi_s-s\phi_s^2-s\phi\phi_{ss}}{\phi^2(\phi-s\phi_s)(\phi-s\phi_s+(r^2-s^2)\phi_{ss})},\quad \rho_3=-\frac{\phi_{ss}}{\phi(\phi-s\phi_s)(\phi-s\phi_s+(r^2-s^2)\phi_{ss})}.$$

    \medskip

For our subsequent investigation, we introduce the covector
\[
m_i = x_i - \frac{s}{u} y_i,
\]
which will play a fundamental role in the expressions of various geometric objects.

\begin{remark}
The covector \( m_i \) satisfies the property \( m_i \neq 0 \). Indeed, suppose that \( m_i = 0 \). Then we have
\[
x_i - \frac{s}{u} y_i = 0.
\]
Differentiating  the above equation  with respect to \( y^j \), we obtain
\[
\frac{s}{u} \left( \delta_{ij} - \frac{1}{u^2} y_i y_j \right) = 0.
\]
Contracting this equation with \( \delta^{ij} \) yields
\[
\frac{s}{u} (n - 1) = 0.
\]
Since \( n \neq 1 \), it follows that \( s = 0 \), and hence \( x_i = 0 \), which contradicts the assumption that \( x \in \mathbb{B}^n(r_0) \setminus \{0\} \). Therefore, \( m_i \neq 0 \). Similarly, one can show that \( s \neq 0 \) and \( r^2 - s^2 \neq 0 \).
\end{remark}

The angular metric of the Kronecker delta is defined by 
  $$\hbar_{ij} = \delta_{ij} - \frac{1}{u^2} y_i y_j.$$

  In what follow, for simplicity, we use the notations $ m^2:=r^2-s^2$ and   $\mu:=\sigma_1$. That is, $\mu_s=(\sigma_1)_s$, $\mu_{ss}=(\sigma_1)_{ss}$. 
  We have the following identities: 
 \begin{equation}
 \label{Identites_sigmas}
 (\sigma_0)_s=\sigma_2, \,\,\, (\sigma_2)_s=-s\mu_s, \,\,\, (\sigma_3)_s=s^2\mu_s-\sigma_2, \,\, \sigma_3=-s\sigma_2.
 \end{equation}
 
 \begin{equation}
 \label{Identites_m^2}
 m^2=m^im_i=x^im_i, \quad x^r\hbar_{ri}=m_i, \quad y^im_i=0, \quad y_im^i=0.
 \end{equation}
 
 \begin{equation}
 \label{Identites_hbar}
\frac{\partial  }{\partial y^k} \frac{y_i}{u}=\frac{1}{u}(\delta_{ik}-\frac{1}{u^2}y_iy_k)=\frac{1}{u}\hbar_{ik}.
 \end{equation}

\section{The Cartan tensor }

~\par
  
 In this section, we investigate key geometric quantities associated with spherically symmetric Finsler metrics, with particular focus on the Cartan tensor and the mean Cartan tensor. Then, we show that  any spherically symmetric Finsler metric with a non-vanishing mean Cartan tensor must be quasi-C-reducible.

  \begin{lemma}\label{C_ijk}
Let $F$ be a spherically symmetric Finsler metric. Then the components of the Cartan tensor, given by $C_{ijk} = \frac{1}{2} \frac{\partial g_{ij}}{\partial y^k}$, can be expressed as
\begin{equation}
\label{Eq:Cartan_tensor}C_{ijk} = \frac{\sigma_2}{2u} \left( \hbar_{ij} m_k + \hbar_{jk} m_i + \hbar_{ik} m_j \right) + \frac{\mu_s}{2u} m_i m_j m_k.
\end{equation}
Moreover, the (1,2)-type form of the Cartan tensor, denoted by $C^r_{jk}$, is given by
 \begin{align}
 \label{Eq:Cartan_1_2}
\nonumber C^r_{jk}=&     \frac{\rho_0  \sigma_2}{2u} \left( \hbar^r_{ j} m_k  + \hbar^r_{k} m_j \right) +\left( \frac{\rho_0  \sigma_2}{2u}  \hbar_{jk}   + \frac{\rho_0 \mu_s}{2u}  m_j m_k\right)  m^r\\
&+ \left(   \frac{\rho_2\sigma_2 m^2}{2u^2}  \hbar_{jk}     +  \frac{2\rho_2\sigma_2+\rho_2 \mu_sm^2}{2u^2}  m_j m_k \right)y^r  \\ 
\nonumber &+   \left(   \frac{\rho_3\sigma_2 m^2}{2u}  \hbar_{jk}     +  \frac{2\rho_3\sigma_2+\rho_3 \mu_sm^2}{2u}  m_j m_k \right)x^r  .
   \end{align}
\end{lemma}

\begin{proof}
Differentiating the expression for $g_{ij}$ in \eqref{Eq:g_ij} with respect to $y^k$, and noting that
\(
\frac{\partial s}{\partial y^k} = \frac{1}{u} m_k,
\)
we obtain:
\begin{align*}
2 C_{ijk} &= \frac{(\sigma_0)_s}{u} \delta_{ij} m_k + \frac{\mu_s}{u} x_i x_j m_k + \frac{(\sigma_2)_s}{u^2} (x_i y_j + x_j y_i) m_k + \frac{\sigma_2}{u} (x_i \hbar_{jk} + x_j \hbar_{ik}) \\
&\quad + \frac{(\sigma_3)_s}{u^3} y_i y_j m_k + \frac{\sigma_3}{u^2} (y_i \hbar_{jk} + y_j \hbar_{ik}).
\end{align*}

Using \eqref{Identites_sigmas}, the above formula  takes the form
\begin{align*}
2 C_{ijk} &= \frac{\sigma_2}{u} \delta_{ij} m_k + \frac{\mu_s}{u} x_i x_j m_k - \frac{s \mu_s}{u^2} (x_i y_j + x_j y_i) m_k + \frac{\sigma_2}{u} (x_i \hbar_{jk} + x_j \hbar_{ik}) \\
&\quad + \frac{s^2 \mu_s - \sigma_2}{u^3} y_i y_j m_k - \frac{s \sigma_2}{u^2} (y_i \hbar_{jk} + y_j \hbar_{ik}).
\end{align*}
Rewriting this in terms of $\hbar_{ij}$ and $m_i$ yields 
\[
2 C_{ijk} = \frac{2 \sigma_2}{u} \left( \hbar_{ij} m_k + \hbar_{jk} m_i + \hbar_{ik} m_j \right) + \frac{2\mu_s}{u} m_i m_j m_k,
\]
which implies the desired result.

Now, using the inverse metric $g^{ij}$  \eqref{Eq:g^ij} together with \eqref{Identites_m^2}, we have  
 \begin{align*}
C^r_{jk}&=g^{ir}C_{ijk}\\
&=(\rho_0\delta^{ir}+ \frac{\rho_1}{u^2}y^iy^r+ \frac{\rho_2}{u} (x^iy^r+x^ry^i)+\rho_3x^ix^r)C_{ijk}\\
&=(\rho_0\delta^{ir}+ \frac{\rho_2}{u} x^iy^r+\rho_3x^ix^r)C_{ijk}\\
&=(\rho_0\delta^{ir}+ \frac{\rho_2}{u} x^iy^r+\rho_3x^ix^r) \left(  \frac{\sigma_2}{2u} \left( \hbar_{ij} m_k + \hbar_{jk} m_i + \hbar_{ik} m_j \right) + \frac{\mu_s}{2u} m_i m_j m_k \right)\\
&=     \frac{\rho_0  \sigma_2}{2u} \left( \hbar^r_{ j} m_k  + \hbar^r_{k} m_j \right) +\left( \frac{\rho_0  \sigma_2}{2u}  \hbar_{jk}   + \frac{\rho_0 \mu_s}{2u}  m_j m_k\right)  m^r\\
&+ \left(   \frac{\rho_2\sigma_2 (r^2-s^2)}{2u^2}  \hbar_{jk}     +  \frac{2\rho_2\sigma_2+\rho_2 \mu_s(r^2-s^2)}{2u^2}  m_j m_k \right)y^r  \\ 
&+   \left(   \frac{\rho_3\sigma_2 (r^2-s^2)}{2u}  \hbar_{jk}     +  \frac{2\rho_3\sigma_2+\rho_3 \mu_s(r^2-s^2)}{2u}  m_j m_k \right)x^r  .
   \end{align*}
   Which completes the proof.    
\end{proof}
\begin{remark}
It should be noted  that, in the above proof, we have used the fact that $y^rC_{rjk}=0$. But  $C^r_{ij} y_r\neq0$ since here we have $y_i=\delta_{ij} y^j$ and  $y_i\neq g_{ij} y^j$. 
\end{remark}

\begin{lemma}
The mean Cartan tensor $C_i:=g^{jk}C_{ijk}$ is given by
\begin{equation}
\label{Eq:C_i}
C_i=\mathcal{A}\,  m_i,
\end{equation}
where 
\begin{equation}
\label{Eq:A_Quasi}
\mathcal{A}:= \frac{1 }{2u}  \left( \rho_0\sigma_2 (n+1)    +  \rho_0\mu_s  m^2      + 3\rho_3\sigma_2  m^2    +  \rho_3\mu_s  m^4 \right).
\end{equation}
\end{lemma}
\begin{proof}
Using the inverse metric tensor \eqref{Eq:g^ij} together with the Cartan tensor \eqref{C_ijk}, we have
\begin{align*}
C_k=&g^{ij}C_{ijk}\\
=&\left( \rho_0\delta^{ij}+ \frac{\rho_1}{u^2}y^iy^j+ \frac{\rho_2}{u} (x^iy^j+x^jy^i)+\rho_3x^ix^j \right)C_{ijk}\\
=&\left( \rho_0\delta^{ij}+\rho_3x^ix^j \right)C_{ijk}\\
=&\left( \rho_0\delta^{ij}+\rho_3x^ix^j \right)\left( \frac{\sigma_2}{2u} \left( \hbar_{ij} m_k + \hbar_{jk} m_i + \hbar_{ik} m_j \right) + \frac{\mu_s}{2u} m_i m_j m_k  \right)\\
=&  \frac{\rho_0\sigma_2 (n+1) }{2u}   m_k  + \frac{\rho_0\mu_s m^2}{2u}   m_k+\frac{3\rho_3\sigma_2 m^2 }{2u}   m_k  + \frac{\rho_3\mu_s m^4}{2u}   m_k\\
=&\frac{1 }{2u}  \left( \rho_0\sigma_2 (n+1)    +  \rho_0\mu_s m^2     + 3\rho_3\sigma_2 m^2   +  \rho_3\mu_s m^4 \right)   m_k\\
=& \mathcal{A} \, m_k.
\end{align*}
  \end{proof}

\begin{definition}
Let $(M,F)$ be a Finsler space of dimension $n \geq 3$. The space is called \emph{quasi-C-reducible} if its Cartan tensor $C_{ijk}$ satisfies the relation:
\begin{equation}\label{paper2.8}
C_{ijk} = Q_{ij} C_k + Q_{jk} C_i + Q_{ki} C_j,
\end{equation}
where $Q_{ij}$ is a symmetric indicatory tensor.
\end{definition}

\begin{theorem}
Let $F = u\,\phi(r,s)$ be a spherically symmetric Finsler metric   of dimension $n \geq 3$.     If the mean Cartan tensor  is is non-zero, then $F$ is quasi-C-reducible.
\end{theorem}

\begin{proof}
For spherically symmetric Finsler metrics with non-zero mean Cartan tensor ($\mathcal{A}\neq 0$), we can write
$$m_i=\frac{1}{\mathcal{A}}C_i.$$
Now, we have

 \begin{align*}
C_{ijk}=&  \frac{\sigma_2}{2u\mathcal{A}} \left( \hbar_{ij} C_k + \hbar_{jk} C_i + \hbar_{ik} C_j \right) + \frac{(\sigma_1)_s}{2u\mathcal{A}^3} C_i C_j C_k  \\
=&\left( \frac{\sigma_2}{2u\mathcal{A}}   \hbar_{ij}+ \frac{(\sigma_1)_s}{6u\mathcal{A}^3} C_i C_j\right) C_k + \left( \frac{\sigma_2}{2u\mathcal{A}}   \hbar_{jk}+ \frac{(\sigma_1)_s}{6u\mathcal{A}^3} C_j C_k\right) C_i+\left( \frac{\sigma_2}{2u\mathcal{A}}   \hbar_{ki}+ \frac{(\sigma_1)_s}{6u\mathcal{A}^3} C_k C_i\right) C_j\\
=&Q_{ij}C_k+Q_{jk}C_i+Q_{ki}C_j,
\end{align*}
where $$Q_{ij}:=\frac{\sigma_2}{2u\mathcal{A}}   \hbar_{ij}+ \frac{(\sigma_1)_s}{6u\mathcal{A}^3} C_i C_j.$$
This completes the proof.
\end{proof}
\begin{remark}
It should be noted that a positive definite Finsler metric with vanishing mean Cartan tensor is necessarily Riemannian. However, there exist non-positive definite Finsler metrics with zero mean Cartan tensor that are not Riemannian; see \cite{deicke,YoussefElgendi2019} for details.
\end{remark}

  \section{The T-tensor}
 ~\par 
 
 The T-tensor plays a fundamental role in the study of Finsler geometry, as it reflects the deviation from Riemannian behavior and provides insight into the underlying curvature and torsion properties of the space.  
In this section, we derive an explicit expression for the T-tensor associated with a spherically symmetric Finsler metric. We provide two examples, namely, spherically symmetric Finsler metric of Kropina-type and Randers-type.
  
  \medskip
  
  For a Finsler manifold $(M,F)$, the $T$-tensor is defined by \cite{ttensor}
\begin{equation} \label{T-tensor}
T_{hijk}=FC_{hijk}-F(C_{rij}C^{r}_{hk}+C_{rjh}C^{r}_{ik}+C_{rih}C^{r}_{jk})
+C_{hij}\ell_k+C_{hik}\ell_j +C_{hjk}\ell_i+C_{ijk}\ell_h,
\end{equation}
where $\ell_j:=\dot{\partial}_j F$, $ C_{rijk}:=\dot{\partial}_rC_{ijk}$ and  $\dot{\partial}_j$ is the differentiation with respect to $y^j$.
The $T$-tensor is totally symmetric in all of its indices.

\begin{proposition}\label{Lem:T-tensor_SSFM}
The T-tensor of    a spherically symmetric Finsler metric   
$F=u\phi(r,s)$   is given by:
\begin{align*}
T_{hijk}=&    \Phi (\hbar_{hi}\hbar_{jk}+\hbar_{hj}\hbar_{ik}+\hbar_{hk}\hbar_{ij})\\
   &+\Psi(\hbar_{hk}m_im_j+\hbar_{hj}m_im_k+\hbar_{hi}m_jm_k+\hbar_{ij}m_hm_k+\hbar_{jk}m_im_h+\hbar_{ik}m_jm_h)\\
    &+\Omega\, m_hm_im_jm_k \\        
\end{align*}
where  
\begin{equation}
\label{Eq:PhiFactored}
\Phi = -\frac{\phi\, \sigma_2}{4u} \left(2 s +  \sigma_2 m^2 \kappa  \right),
\end{equation}
  \begin{equation}
\label{Eq:PsiFactored}
\Psi = \frac{\phi  }{4u} \left(  \frac{4\phi_s\, \sigma_2}{\phi}
-    2 s \mu_s 
-    2\rho_0 \sigma_2^2 - \sigma_2 \kappa (2 \sigma_2 +\mu_s m^2)   \right),
\end{equation}
    
\begin{equation}
\label{Eq:OmegaFactored}
\Omega =  \frac{\phi  }{4u} \left(   \frac{8\mu_s \phi_s}{\phi }+  
2 \mu_{ss}
- 6\rho_0 \sigma_2 \mu_s -3 (2\sigma_2 + \mu_s m^2)(\kappa \mu_s + 2\rho_3 \sigma_2)
\right),
\end{equation} 
and $\kappa:=\rho_0 + \rho_3m^2$.
\end{proposition}

To simplify and shorten the proof of the above proposition, we present the following lemmas.

  \begin{lemma}
For a spherically symmetric Finsler metric  
$F=u\phi(r,s)$,    the vertical derivative of the Cartan tensor \( C_{ijk} \) can be expressed as
\begin{eqnarray}\label{lamda}
  \dot{\partial}_h{C}_{ijk}  &=&
\nonumber     -\frac{\sigma_2}{2u^2}(\hbar_{ik}n_{jh}+\hbar_{jk}n_{ih}+\hbar_{ij}n_{kh}
   +\hbar_{jh}n_{ik}+\hbar_{kh}n_{ij}+\hbar_{ih}n_{jk}) \\
\nonumber    &&-\frac{s\sigma_2}{2u^2}(\hbar_{ik}\hbar_{jh}+\hbar_{jk}\hbar_{ih}+\hbar_{kh}\hbar_{ij})
   -\frac{s\mu_s}{2u^2}(\hbar_{ij}m_hm_k+\hbar_{ki}m_jm_h+\hbar_{hk}m_im_j\\
\nonumber    &&+\hbar_{jh}m_km_i+\hbar_{ih}m_jm_k+\hbar_{kj}m_im_h)+\frac{\mu_{ss}}{2u^2}m_im_jm_hm_k\\
   &&-\frac{\mu_s}{2u^2}(n_{ij}m_hm_k+n_{kh}m_im_j)
\end{eqnarray}
where  $n_{ij}:=\frac{1}{u}(y_im_j+y_jm_i)$.  
  \end{lemma}

 \begin{proof}
The result follows by using Lemma \ref{C_ijk} and making use of the fact that $\dot{\partial}_is=\frac{m_i}{u}$ together with  applying the following identities:
\[
\frac{\partial u}{\partial y^k} = \frac{1}{u} y_k, \quad
\frac{\partial m_k}{\partial y^j} = -\frac{1}{u^2} y_j m_k + \frac{s}{u^3} y_j y_k - \frac{s}{u} \delta_{jk}, \quad
\frac{\partial \hbar_{ij}}{\partial y^k} = \frac{2}{u^4} y_i y_j y_k - \frac{1}{u^2} y_i \delta_{jk} - \frac{1}{u^2} y_j \delta_{ik}.
\]
\end{proof}

 \begin{lemma}
 For a spherically symmetric Finsler metric  
$F=u\phi(r,s)$,  the following identity holds for the contraction of Cartan tensors:
  \begin{align*}
  {C}_{ijr} {C}^r_{hk}&+{C}_{jkr} {C}^r_{hi}+{C}_{ikr} {C}^r_{hj}\\
=&\frac{3(2\rho_0\sigma_2 \mu_s+(2\sigma_2+\mu_s  m^2)(\rho_0 \mu_s+2\rho_3\sigma_2+\rho_3 \mu_s m^2))}{4u^2} m_im_jm_km_h\\
&+\frac{\sigma_2^2m^2(\rho_0+\rho_3m^2)}{4u^2} (\hbar_{hk}\hbar_{ij}+\hbar_{hj}\hbar_{ik}+\hbar_{hi}\hbar_{jk})\\
&+\frac{2\rho_0 \sigma_2^2+(2\sigma_2+\mu_s  m^2)(\rho_0\sigma_2+\rho_3\sigma_2m^2)}{4u^2}(\hbar_{ij}m_hm_k+\hbar_{hj}m_im_k\\
&+\hbar_{hi}m_jm_k+\hbar_{hk}m_im_j+\hbar_{jk}m_im_h+\hbar_{ik}m_jm_h)
   \end{align*}
   \end{lemma}

 \begin{proof}
 Using the definitions of the Cartan tensor \eqref{Eq:Cartan_tensor} and \eqref{Eq:Cartan_1_2} together with the identities \eqref{Identites_hbar} and \eqref{Identites_m^2}, we have
   \begin{align*}
C^r_{hk} C_{ijr}&= \Big{[}     \frac{\rho_0  \sigma_2}{2u} \left( \hbar^r_{ h} m_k  + \hbar^r_{k} m_h \right) +\left( \frac{\rho_0  \sigma_2}{2u}  \hbar_{hk}   + \frac{\rho_0 \mu_s}{2u}  m_h m_k\right)  m^r\\
&+ \left(   \frac{\rho_2\sigma_2 m^2}{2u^2}  \hbar_{hk}     +  \frac{2\rho_2\sigma_2+\rho_2 \mu_s m^2}{2u^2}  m_h m_k \right)y^r  \\ 
&+   \left(   \frac{\rho_3\sigma_2 m^2}{2u}  \hbar_{hk}     +  \frac{2\rho_3\sigma_2+\rho_3 \mu_s m^2}{2u}  m_h m_k \right)x^r  \Big{]} \\
&\times  \Big{(}   \frac{\sigma_2}{2u} \left( \hbar_{ij} m_r + \hbar_{jr} m_i + \hbar_{ir} m_j \right) + \frac{\mu_s}{2u} m_i m_j m_r   \Big{)} \\
&=   \frac{\rho_0 \sigma_2^2 }{4u^2} (2\hbar_{ij}m_hm_k+\hbar_{hj}m_im_k+\hbar_{hi}m_jm_k+\hbar_{jk}m_im_h+\hbar_{ik}m_jm_h)\\
&+   \frac{2\rho_0 \sigma_2 \mu_s }{4u^2} m_im_jm_hm_k  \\
& +\left( \frac{\rho_0  \sigma_2}{2u}  \hbar_{hk}   + \frac{\rho_0 \mu_s}{2u}  m_h m_k\right) \left(   \frac{\sigma_2}{2u} \left(m^2 \hbar_{ij}   \right) + \frac{2\sigma_2+\mu_s m^2 }{2u} m_i m_j    \right) \\
&+\left(   \frac{\rho_3\sigma_2 m^2}{2u}  \hbar_{hk}     +  \frac{2\rho_3\sigma_2+\rho_3 \mu_s m^2}{2u}  m_h m_k \right)   \Big{(}   \frac{\sigma_2}{2u} \left( m^2  \hbar_{ij}     \right) \\
&+ \frac{2\sigma_2+\mu_s m^2}{2u} m_i m_j   \Big{)} \\
&=  \frac{\rho_0 \sigma_2^2 }{4u^2} (\hbar_{hj}m_im_k+\hbar_{hi}m_jm_k+\hbar_{jk}m_im_h+\hbar_{ik}m_jm_h)\\
&+\frac{2\rho_0\sigma_2 \mu_s+(2\sigma_2+\mu_s  m^2)(\rho_0 \mu_s+2\rho_3\sigma_2+\rho_3 \mu_s m^2)}{4u^2} m_im_jm_km_h\\
&+\frac{\sigma_2^2 m^2(\rho_0+\rho_3 m^2)}{4u^2} \hbar_{hk}\hbar_{ij}\\
&+\frac{(2\sigma_2+\mu_s  m^2)(\rho_0\sigma_2+\rho_3\sigma_2 m^2)}{4u^2}\hbar_{hk}m_im_j\\
&+\frac{(2\sigma_2+\mu_s  m^2)(\rho_0\sigma_2+\rho_3\sigma_2 m^2)}{4u^2}\hbar_{ij}m_hm_k.
   \end{align*}
Now, by applying the cyclic sum over the indices \( i, j, k \) to the expression \( C^r_{hk} C_{ijr} \), we obtain the desired formula.
 \end{proof}

  \begin{lemma}
Let  $F=u\phi(r,s)$ be  a spherically symmetric Finsler metric,  the following identity holds for the Cartan tensor \( C_{ijk} \) and the normalized supporting element \( \ell_i \):
\begin{align*}
C_{hij} \ell_k + C_{hik} \ell_j + C_{hjk} \ell_i + C_{ijk} \ell_h 
=&\ \frac{\mu_s \phi}{2u} \left( m_i m_j n_{kh} + m_k m_h n_{ij} \right)
+ \frac{\mu_s \phi_s}{2u} m_i m_j m_h m_k \\
&+ \frac{\phi_s \sigma_2}{u} \big(
h_{ij} m_h m_k + h_{ki} m_j m_h + h_{hk} m_i m_j \\
&\quad + h_{jh} m_k m_i + h_{ih} m_j m_k + h_{kj} m_i m_h \big) \\
&+ \frac{\phi \sigma_2}{2u} \big(
h_{ik} n_{jh} + h_{jk} n_{ih} + h_{ij} n_{kh} \\
&\quad + h_{jh} n_{ik} + h_{kh} n_{ij} + h_{ih} n_{jk} \big).
\end{align*}
\end{lemma}

   \begin{proof}
  Since $ \ell_i:=\dot{\partial}_iF=\dot{\partial}_i(u\phi)= \frac{\phi}{u}y_i+\phi_s m_i$, we get
   \begin{align*}
   C_{ijk} \ell_h =&\frac{\phi \sigma_2}{2u^2}(\hbar_{ij}m_ky_h+\hbar_{jk}m_iy_h+\hbar_{ik}m_jy_h)\\
   &+\frac{\phi_s \sigma_2}{2u}(\hbar_{ij}m_km_h+\hbar_{jk}m_im_h+\hbar_{ik}m_jm_h)\\
   &+\frac{\mu_s\phi }{2u^2} m_im_jm_ky_h+\frac{ \mu_s\phi_s}{2u} m_im_jm_km_h.
   \end{align*}
   Now, by applying the cyclic sum over the indices \( i, j, k ,h\) to the expression \(C_{ijk} \ell_h \), we obtain the desired formula.
\end{proof}
   
  Now, we are in a position to prove  Proposition \ref{Lem:T-tensor_SSFM}. 

\begin{proof}[Proof of Proposition \ref{Lem:T-tensor_SSFM}]
By using the preparatory lemmas and formulas established above, we proceed as follows:
\begin{align*}
T_{hijk}=&FC_{hijk}-F(C_{rij}C^{r}_{hk}+C_{rjh}C^{r}_{ik}+C_{rih}C^{r}_{jk})
+C_{hij}\ell_k+C_{hik}\ell_j +C_{hjk}\ell_i+C_{ijk}\ell_h\\
=&  -\frac{\sigma_2 \phi}{2u}(\hbar_{ik}n_{jh}+\hbar_{jk}n_{ih}+\hbar_{ij}n_{kh}
   +\hbar_{jh}n_{ik}+\hbar_{kh}n_{ij}+\hbar_{ih}n_{jk}) \\
       &-\frac{s\sigma_2 \phi  }{2u}(\hbar_{ik}\hbar_{jh}+\hbar_{jk}\hbar_{ih}+\hbar_{kh}\hbar_{ij})
   -\frac{s\mu_s \phi}{2u}(\hbar_{ij}m_hm_k+\hbar_{ki}m_jm_h+\hbar_{hk}m_im_j\\
      &+\hbar_{jh}m_km_i+\hbar_{ih}m_jm_k+\hbar_{kj}m_im_h)+\frac{\phi\mu_{ss}}{2u}m_im_jm_hm_k\\
    &-\frac{\mu_s \phi}{2u}(n_{ij}m_hm_k+n_{kh}m_im_j)\\
&-\frac{3 \phi (2\rho_0 \sigma_2 \mu_s+(2\sigma_2+\mu_s  m^2)(\rho_0 \mu_s+2\rho_3\sigma_2+\rho_3 \mu_s m^2))}{4u} m_im_jm_km_h\\
&-\frac{\phi \sigma_2^2m^2(\rho_0+\rho_3(r^2-s^2))}{4u} (\hbar_{hk}\hbar_{ij}+\hbar_{hj}\hbar_{ik}+\hbar_{hi}\hbar_{jk})\\
&-\frac{\phi(2\rho_0 \sigma_2^2+(2\sigma_2+\mu_s  m^2)(\rho_0\sigma_2+\rho_3\sigma_2m^2))}{4u}(\hbar_{ij}m_hm_k+\hbar_{hj}m_im_k\\
&+\hbar_{hi}m_jm_k+\hbar_{hk}m_im_j+\hbar_{jk}m_im_h+\hbar_{ik}m_jm_h)\\
& +  \frac{\mu_s\phi }{2u}(m_im_jn_{kh}+m_km_hn_{ij})
 +\frac{2 \mu_s\phi_s}{u} m_im_jm_hm_k\\
    & +\frac{\phi_s \sigma_2}{u}(h_{ij}m_hm_k+h_{ki}m_jm_h+h_{hk}m_im_j+h_{jh}m_km_i+h_{ih}m_jm_k+h_{kj}m_im_h)\\
     & 
   +\frac{\phi \sigma_2}{2u}   (h_{ik}n_{jh}+h_{jk}n_{ih}+h_{ij}n_{kh}
   +h_{jh}n_{ik}+h_{kh}n_{ij}+h_{ih}n_{jk})\\
=&    \Phi (h_{hi}h_{jk}+h_{hj}h_{ik}+h_{hk}h_{ij})\\
   &+\Psi(h_{hk}m_im_j+h_{hj}m_im_k+h_{hi}m_jm_k+h_{ij}m_hm_k+h_{jk}m_im_h+h_{ik}m_jm_h)\\
    &+\Omega\, m_hm_im_jm_k \\      
\end{align*}

where we set 

\begin{equation}
\label{Eq:Phi}
\Phi:= -\frac{s\sigma_2 \phi  }{2u}  -\frac{\phi \sigma_2^2m^2(\rho_0+\rho_3m^2)}{4u} 
\end{equation}
  
  \begin{equation}
\label{Eq:Psi}
\Psi:= 
-\frac{s\mu_s \phi}{2u}-\frac{\phi(2\rho_0 \sigma_2^2+(2\sigma_2+\mu_s  m^2)(\rho_0\sigma_2+\rho_3\sigma_2m^2))}{4u}+\frac{\phi_s \sigma_2}{u}
\end{equation}
  
  \begin{equation}
\label{Eq:Omega}
\Omega:= 
 \frac{\mu_{ss}\phi}{2u}-\frac{3 \phi (2\rho_0 \sigma_2 \mu_s+(2\sigma_2+\mu_s  m^2)(\rho_0 \mu_s+2\rho_3\sigma_2+\rho_3 \mu_s m^2))}{4u}+\frac{2 \mu_s\phi_s}{u}
\end{equation}

Simplifying the above formuale of $\Phi$, $\Psi$ and $\Omega$ we get the formulae given in \eqref{Eq:PhiFactored}, \eqref{Eq:PsiFactored}, and \eqref{Eq:OmegaFactored}.
\end{proof}

For a spherically symmetric Finsler metric of the form \(F = u\, \phi(r,s)\), one can explicitly compute the tensors \(\Phi\), \(\Psi\), and \(\Omega\) that appear in the expression of the \(T\)-tensor. These calculations can be performed either by direct symbolic computation or efficiently using computer algebra systems such as Maple. For illustration, we present the following two examples providing explicit formulas for the \(T\)-tensor. The Maple worksheet  and corresponding PDF files for the calculations can be found at:\\
\url{https://github.com/salahelgendi/Calculations_T-tensor_SSFM.git}

\begin{example}
Let $F$ be a spherically symmetric Finsler metric of  Kropina-type, that is, 
\[
F = \frac{u}{s}, \quad \text{with } \phi(r,s) = \frac{1}{s}.
\]
The \(T\)-tensor of $F$ is given by
\begin{align*}
T_{hijk} &= \frac{2}{s u  r^{2} } \bigl(\hbar_{hi} \hbar_{jk} + \hbar_{hj} \hbar_{ik} + \hbar_{hk} \hbar\hbar_{ij}\bigr) \\
&\quad + \frac{2}{u r^{2} s^{3}} \bigl(\hbar_{hi} m_j m_k + \hbar_{hj} m_i m_k + \hbar_{ij} m_h m_k + \hbar_{jk} m_i m_h + \hbar_{hk} m_i m_j + \hbar_{ik} m_j m_h\bigr) \\
&\quad + \frac{6}{u r^{2} s^{5}} m_h m_i m_j m_k.
\end{align*}
\end{example}

\begin{example}
Similarly, let  $F$ be a spherically symmetric Finsler metric of  Randers-type, that is, 

\[
F = u (1 + s), \quad \text{with } \phi(r,s) = 1 + s.
\]
T he \(T\)-tensor of $F$ is  given by
\[
T_{hijk} = - \frac{r^{2} + s^{2} + 2 s}{4 u} \bigl(\hbar_{hi} \hbar_{jk} + \hbar_{hj} \hbar_{ik} + \hbar_{hk} \hbar\hbar_{ij}\bigr).
\]
\end{example}

It is worth noting that the \(T\)-tensor of the Kropina metric has also been obtained by Shibata \cite{r2.12,beta2}, and the \(T\)-tensor of the Randers metric was studied by Matsumoto \cite{r2.8}.

\section{T-condition}
In this section, we focus on the \(T\)-condition and provide a complete characterization of all spherically symmetric Finsler metrics that satisfy this condition. Our goal is to analyze the structural constraints imposed by the \(T\)-condition on metrics of the form \(F = u\, \phi(r,s)\), thereby identifying precisely which functions \(\phi\) yield metrics with vanishing or specially structured \(T\)-tensors. This study contributes to a deeper understanding of the geometric properties and special classes within the family of spherically symmetric Finsler metrics.

Before proceeding, we present the following proposition that will play a key role in our subsequent analysis.
 \begin{proposition}\label{Prop:sigma_2}
A spherically symmetric Finsler metric $F = u\,\phi(r,s)$ is Riemannian if and only if $\sigma_2 = 0$.
\end{proposition}

\begin{proof}
Let $F = u\,\phi(r,s)$ be a spherically symmetric Finsler metric. Then, we have
\begin{align*}
\sigma_2 = 0 &\;\Longleftrightarrow\; (\phi - s\phi_s)\phi_s - s\phi\,\phi_{ss} = 0 \\
&\;\Longleftrightarrow\; s(\phi\,\phi_s)_s = \phi\,\phi_s \\
&\;\Longleftrightarrow\; \frac{(\phi\,\phi_s)_s}{\phi\,\phi_s} = \frac{1}{s} \\
&\;\Longleftrightarrow\; \phi\,\phi_s = c_1(r)\,s \quad \text{(by integration)} \\
&\;\Longleftrightarrow\; \phi^2 = c_1(r)\,s^2 + c_2(r) \\
&\;\Longleftrightarrow\; \phi = \sqrt{c_1(r)\,s^2 + c_2(r)} \\
&\;\Longleftrightarrow\; F \text{ is Riemannian.}
\end{align*}
Here, $c_1(r)$ and $c_2(r)$ are arbitrary smooth functions of $r$.
\end{proof}

To investigate spherically symmetric Finsler metrics that satisfy the T-condition, that is, those for which the T-tensor vanishes, we begin by solving the associated differential equations. To facilitate this, we introduce a function $W(r,s)$ defined as follows:
$$W(r,s):=\frac{\phi_s}{\phi-s\phi_s} .$$

The function $W(s)$ serves to simplify and facilitate the solution of the  differential equations addressed in this section. Furthermore, the function $\phi$ is given by
\begin{equation}\label{EQ:W_1}
\phi(r,s)=\exp\left(\int_0^s \frac{W}{1+tW}dt\right).
\end{equation}

\begin{proposition}\label{Prop:Phi=0}
The following identity holds: 
$$2s+m^2\sigma_2\kappa=\frac{(\phi-s\phi_s)^2}{\phi(\phi-s\phi_s+m^2\phi_{ss})}\left(  W_s+\left(\frac{1}{s}+\frac{2s}{m^2}\right)W+\frac{2}{m^2}\right).$$
\end{proposition}

\begin{proof}
Using the formulae of $\sigma_2$ and $\kappa$, we have
\begin{align*}
2s+m^2\sigma_2\kappa=&  2s+\frac{m^2((\phi-s\phi_s)\phi_s-s\phi\phi_{ss})}{\phi(\phi-s\phi_s+m^2\phi_{ss})}\\
=& \frac{2s(\phi(\phi-s\phi_s+m^2\phi_{ss}))+m^2((\phi-s\phi_s)\phi_s-s\phi\phi_{ss})}{\phi(\phi-s\phi_s+m^2\phi_{ss})}\\
=& \frac{2s\phi(\phi-s\phi_s)+m^2(\phi-s\phi_s)\phi_s}{\phi(\phi-s\phi_s+m^2\phi_{ss})}\\
=&\frac{(\phi-s\phi_s)^2}{\phi(\phi-s\phi_s+m^2\phi_{ss})}\left( \frac{2s \phi}{\phi-s\phi_s}+\frac{m^2 \phi_s}{\phi-s\phi_s}+\frac{sm^2\phi\phi_{ss}}{(\phi-s\phi_s)^2}\right)\\
=&\frac{(\phi-s\phi_s)^2}{\phi(\phi-s\phi_s+m^2\phi_{ss})}\left(  W_s+\left(\frac{1}{s}+\frac{2s}{m^2}\right)W+\frac{2}{m^2}\right).
\end{align*}
\end{proof}

\begin{theorem}\label{Theorem_T-condition}
A spherically symmetric Finsler metric $F=u\phi(r,s)$ with $n \geq 3$ satisfies the $T$-condition if and only if either the metric is Riemannian or    $\phi$ is given by
\begin{equation}\label{berwald}
\phi(r,s) =
a(r) \, s^{\frac{c(r)  r^2 - 1}{c(r)  r^2}} ( r^2 -  s^2)^{\frac{1}{2 c(r) r^2}}.
\end{equation}
\end{theorem}

\begin{proof}
Assume that a spherically symmetric Finsler metric $F=u\phi(r,s)$ satisfying the T-condition, that is,  $T_{hijk}=0$.  Then we have
 \begin{align}
 \label{T=0}
\nonumber &  \Phi (\hbar_{hi}\hbar_{jk}+\hbar_{hj}\hbar_{ik}+h_{hk}\hbar_{ij})\\
   &+\Psi(\hbar_{hk}m_im_j+\hbar_{hj}m_im_k+\hbar_{hi}m_jm_k+\hbar_{ij}m_hm_k+\hbar_{jk}m_im_h+\hbar_{ik}m_jm_h)\\
 \nonumber   &+\Omega\, m_hm_im_jm_k  =0  .
\end{align}
Contracting \eqref{T=0} by $x^h$, we get
\begin{equation}
\label{Eq_1}
A (\hbar_{jk}m_i+\hbar_{ik}m_j+\hbar_{ij}m_k)+B m_im_jm_k=0.
\end{equation}
  Where $A$ and $B$ are scalar functions defined by
  $$A:=\Phi+m^2\Psi, \quad B:=3\Psi+m^2\Omega$$
  Contracting \eqref{Eq_1} by $x^i x^j$, we get
  $$3A m^2  m_k+ B m^4 m_k=0.$$
  Since $m^2\neq0$ and $m_k\neq0$,  we have    
  \begin{equation}
  \label{Eq:T-condition_1}
 3 A+m^2 B=0.
  \end{equation}
  
  Similarly,  contracting \eqref{Eq_1}  by $\delta^{ij}$, we get

   \begin{equation}
  \label{Eq:T-condition_2}
 (n+1) A+m^2 B=0.
  \end{equation}
  Solving \eqref{Eq:T-condition_1} and \eqref{Eq:T-condition_2} for $A$ and $B$ using the fact that $n\geq3$, we get 
  $$A=0, \quad B=0.$$
  
 Therefore, we have
\begin{eqnarray}\label{T=0_1}
  \Phi+m^2\Psi=0,\quad 3\Psi+m^2\Omega=0.
\end{eqnarray}

Again, contraction \eqref{T=0} by $\delta^{hi}$, we obtain
$$
  ((n+1)\Phi+m^2\Psi) \hbar_{jk}+((n+3)\Psi+m^2\Omega)m_jm_k=0.
$$
Following a similar procedure as before, contacting  the above equation by $x^jx^k$, we get 
\begin{equation}
\label{Eq_3}
\mathcal{A}  +m^2\mathcal{B}=0
\end{equation}
where $\mathcal{A}:=(n+1)\Phi+m^2\Psi$, $\mathcal{B}:=(n+3)\Psi+m^2\Omega
$. 
Also, contraction by $\delta^{jk}$ yields 
\begin{equation}
\label{Eq_4}
(n-1)\mathcal{A}  +m^2\mathcal{B}=0.
\end{equation}
Now, solving  \eqref{Eq_3} and \eqref{Eq_4},  taking into account that $n\geq 3$ into account, we have
\begin{eqnarray}\label{T=0_2}
 \mathcal{A}= (n+1)\Phi+m^2\Psi=0, \quad \mathcal{B}=(n+3)\Psi+m^2\Omega=0.
\end{eqnarray}
 Now, solving the equations \eqref{T=0_1} and \eqref{T=0_2} for $\Phi$, $\Psi$ and $\Omega$, we have $\Phi=0$, $\Psi=0$ and $\Omega=0$.

That is, we have $\Phi = 0$. Hence, we have either $\sigma_2=0$ or  $2s+m^2\sigma_2\kappa=0$.
By Proposition \ref{Prop:sigma_2}, if $\sigma_2=0$, then $F$ is Riemannian. Now, let $2s+m^2\sigma_2\kappa=0$,
taking into account that $\phi - s\phi_s \neq 0$, then we have    
\[
W_s + \left( \frac{1}{s} + \frac{2s}{m^2} \right)W = -\frac{2}{m^2}.
\]
This is a first-order linear differential equation which has the solution
\[
W = \frac{c b^2 - 1}{s} - c s = \frac{c(b^2 - s^2) - 1}{s}, \quad c \text{ is a constant}.
\]
Hence,
\[
1 + sW = c b^2 - c s^2,
\qquad
\frac{W}{1 + sW} = \frac{1}{s} - \frac{1}{c s(b^2 - s^2)}.
\]

Using equation~\eqref{EQ:W_1}, the function $\phi(s)$ is given by \eqref{berwald}.

\medskip

Conversely, let $\Phi=0$, then we have either $\sigma_2=0$ or $2s+\kappa \sigma_2 m^2=0$. If $\sigma_2=0$, then  the space is Riemannian,  and hence $\Psi=0$ and $\Omega=0$. And if  $2s+\kappa \sigma_2 m^2=0$, then we get the explicit formula \eqref{berwald} of $\phi$. Now,  substituting this expression into $\Psi$ and $\Omega$ yields $\Psi = 0$ and $\Omega = 0$.    Calculations to verify the vanishing of $\Psi$ and $\Omega$ are available at\\
\url{https://github.com/salahelgendi/Calculations_T-tensor_SSFM.git}
\end{proof}

   \bigskip

\noindent\textbf{Declarations}\\
\begin{itemize}
\item \textbf{Competing interests}: The author  declares no conflict of interest.
  \item \textbf{Availability of data and material}: Not applicable.
  \item \textbf{Funding}: Not applicable.
\end{itemize}

\providecommand{\bysame}{\leavevmode\hbox
to3em{\hrulefill}\thinspace}
\providecommand{\MR}{\relax\ifhmode\unskip\space\fi MR }
\providecommand{\MRhref}[2]{%
  \href{http://www.ams.org/mathscinet-getitem?mr=#1}{#2}
} \providecommand{\href}[2]{#2}


\begin{thebibliography}{10}
 
\bibitem{Asanov1985}
G. S. Asanov, \emph{Finsler Geometry, Relativity and Gauge Theories}, D. Reidel Publishing Company, 1985.



\bibitem{deicke}
A. Deicke, \emph{$\ddot{U}$ber die Finsler-r$\ddot{a}$ume mit $A_i=0$}, Arch. Math., \textbf{4} (1953), 45--51.

\bibitem{Elgendi2021}
S. G. Elgendi, \emph{Solutions for the Landsberg unicorn problem in Finsler geometry}, J. Geom. Phys., \textbf{159}, (2021). arXiv:1908.10910 [math.DG].

\bibitem{Elgendi2021-SSM}
S. G. Elgendi, \textit{On the classification of Landsberg spherically symmetric Finsler metrics}, Int. J. Geom. Methods Mod. Phys., \textbf{18} (2021).

\bibitem{Elgendi2023-SSM}
S. G. Elgendi, \textit{A note on "On the classification of Landsberg spherically symmetric Finsler metrics"}, Int. J. Geom. Methods Mod. Phys., \textbf{20} (2023), 2350096.

\bibitem{Elgendi-LBp}
S. G. Elgendi, \emph{On the problem of non-Berwaldian Landsberg spaces}, Bull. Aust. Math. Soc., \textbf{102}, (2020), 331--341.

\bibitem{Elgendi-ST_condition}
S. G. Elgendi and L. Kozma, \emph{$(\alpha,\beta)$-metrics satisfying T-condition or $\sigma$T-condition}, J. Geom. Anal. (2020).


\bibitem{Tayebi-et.}
T. Khani-Moghaddam, M. Rafie-Rad, and A. Tayebi, \emph{The Landsberg curvature of the class of spherically symmetric Finsler metrics}, To appear in \emph{Int. J. Geom. Methods Mod. Phys.}, (2025).


\bibitem{MoZhou2009}
X. Mo and L. Zhou, \textit{Spherically symmetric Finsler metrics with constant Ricci curvature}, Publ. Math. Debrecen \textbf{74} (2009), no. 1--2, 111--127.

\bibitem{Zhou_Mo}
X. Mo and L. Zhou, \emph{The curvatures of spherically symmetric Finsler metrics in \(\mathbb{R}^n\)}, arXiv:1202.4543 [math.DG], 2012.

\bibitem{MoZhouZhu2010}
X. Mo, L. Zhou, and C. Zhu, \textit{Spherically symmetric Finsler metrics with constant flag curvature}, J. Aust. Math. Soc. \textbf{88} (2010), 41--55.


\bibitem{r2.8}
M. Matsumoto, \emph{On Finsler spaces with Randers metric and special forms of important tensors}, J. Math. Kyoto Univ., \textbf{14} (1974), 477--498.

\bibitem{ttensor}
M. Matsumoto, \emph{V-transformations of Finsler spaces. I. Definition, infinitesimal transformations and isometries}, J. Math. Kyoto Univ., \textbf{12} (1972), 479--512.



\bibitem{Pfeifer2012}
C. Pfeifer and M. Wohlfarth, \emph{Finsler geometric extension of Einstein gravity}, Phys. Rev. D, \textbf{85} (2012), 064009.


\bibitem{Rutz1993}
S. F. Rutz, \emph{A Finsler generalization of Einstein's vacuum field equations}, Gen. Relativity Gravitation, \textbf{25} (1993), 1139--1158.


\bibitem{r2.12}
C. Shibata, \emph{On Finsler spaces with Kropina metric}, Rep. Math. Phys., \textbf{13} (1978), 117--128.




\bibitem{Vacaru2012}
S. Vacaru, \emph{Modified gravity and Finsler geometry}, Int. J. Mod. Phys. D, \textbf{21} (2012), 1250072.

\bibitem{beta2}
Nabil L. Youssef, S. H. Abed and S. G. Elgendi, \emph{Generalized $\beta$-conformal change and special Finsler spaces}, Int. J. Geom. Methods Mod. Phys., \textbf{9 (3)}, (2012), 1250016, 25 pages.

\bibitem{YoussefElgendi2019}
Nabil L. Youssef, S. G. Elgendi, and E. H. Taha, \emph{Semi-concurrent vector fields in Finsler geometry}, Differ. Geom. Appl. \textbf{65} (2019), 1--15.

 
 

\end{thebibliography}
\end{document}